\documentclass[11pt]{article}
\usepackage{graphicx}
\usepackage{amsfonts}
\usepackage{amsmath}
\usepackage{amssymb}
\usepackage{url}
\usepackage[colorlinks=true,citecolor=blue]{hyperref}
\usepackage{fancyhdr}
\usepackage{indentfirst}
\usepackage{enumerate}
\usepackage{amsthm}
\usepackage{color}
\usepackage{comment}
\usepackage[numbers]{natbib}

\def\d{\mathrm{d}}
\def\laweq{\buildrel d \over =}

\newcommand{\VaR}{\mathrm{VaR}}

\newcommand{\WVaR}{\overline{\mathrm{VaR}}}
\newcommand{\BVaR}{\underline{\mathrm{VaR}}}

\newcommand{\U}{\mathrm{U}}
\newcommand{\X}{\mathbf{X}}

\newcommand{\E}{\mathbb{E}}
\newcommand{\R}{\mathbb{R}}
\newcommand{\G}{\mathbb{G}}
\newcommand{\N}{\mathbb{N}}
\newcommand{\M}{\mathcal{M}}
\newcommand{\p}{\mathbb{P}}

\newcommand{\id}{\mathrm{I}}
\newcommand{\lcx}{\prec_{\mathrm{cx}}}

\renewcommand{\ge}{\geqslant}
\renewcommand{\le}{\leqslant}

\renewcommand{\leq}{\leqslant}
\renewcommand{\epsilon}{\varepsilon}
\renewcommand{\sc}{ }
\newcommand{\esssup}{\mathrm{ess\mbox{-}sup}}

\theoremstyle{plain}
\newtheorem{theorem}{Theorem}

\newtheorem{proposition}[theorem]{Proposition}
\theoremstyle{definition}
\newtheorem{definition}{Definition}

\newtheorem{remark}{Remark}
\newtheorem{problem}{Open Problem}

\begin{document}

 \title{Current Open Questions in Complete Mixability}
 \author{Ruodu Wang\thanks{Department of Statistics and Actuarial Science, University of Waterloo, Waterloo, ON N2L 3G1, Canada. Email: \url{wang@uwaterloo.ca}.
}}
 \maketitle
\begin{abstract}
Complete and joint mixability has raised considerable interest in recent few years, in both the theory of distributions with given margins, and applications in discrete optimization and quantitative risk management.
We list various open questions in the theory of complete and joint mixability, which are mathematically concrete, and yet accessible to a broad range of researchers without specific background knowledge.
In addition to the discussions on open questions, some results contained in this paper are new.

\begin{bfseries}Key-words\end{bfseries}: complete mixability; joint mixability; dependence; optimization; Fr\'echet problems

\begin{bfseries}AMS 2010 subject classifications:\end{bfseries} 60C05; 60E15
\end{abstract}

\section{Complete  and joint mixability}\label{sec:intro}

\subsection{Motivation of the paper}

In this paper we summarize the current state-of-the-art of research in complete and joint mixability, and list about a dozen of open questions which the author considers challenging and with  potential impact.
 Many of the open questions emerged via private communications with other researchers\footnote{some names are listed in the Acknowledgement.}.  Most of the questions are formulated in terms of complete mixability due to its nicer mathematical properties, and this is already reflected in the title of this paper.
The paper carries dual purposes: to stimulate  research activities leading to developments in the listed challenges, and to introduce the topic of complete and joint mixability to a broader range of scholars, especially from different fields of study other than applied probability.
Both purposes serve to advance the study of this topic as well as of the expanding field of dependence modeling.

One nice feature of complete and joint mixability is that the topic requires very little background knowledge, and the mathematical questions are still concrete and challenging.
The definitions, existing results and open questions can  be easily understood by any graduate students in analysis,  combinatorics, probability or statistics. This paper hopefully enhances the accessibility of the subject to a broad range of researchers at all levels and from all fields of mathematics and statistics.

Although the main purpose of this paper is to discuss open questions, some useful results listed  in Section 2 (in particular, Theorems \ref{dual}-\ref{norm2}) are original in this paper.

It is the author's sincere hope that after one or a few decades, most of the questions listed in this paper would have been answered,   connected to other fields of mathematics and its applications, and inspired new research directions that are unseen from today.
Some of the questions listed here may be easily answered by some experts, especially those from other fields, due to the obvious limitation of the author's knowledge.
 The opinions expressed in this paper as well as any errors are
solely the responsibility of the author.

\subsection{An optimization problem}\label{sec:opt}

We start the story of mixability by a simple optimization problem.
Suppose that there are $n$ steps in the production of an equipment of a certain type.
A company employs $m$ workers  specialized  in each of its $n$ production steps, that is, $mn$ workers in total; as such the company is able to produce $m$ equipments simultaneously. We use $(i,j)$ for the $i$-th worker employed in the $j$-th production step. Suppose that the time for worker $(i,j)$ to finish her job is a positive number $a_{i,j}$. Each time, the company produces $m$ equipments  and then send them out  to a buyer. Naturally, the company is interested in minimizing the time $T$ of production of the $m$ equipments, namely,  $T=\max\{t_1,\dots,t_m\}$ where $t_i$ stands for the time taken in the production of the $i$-th equipment, $i=1,\dots,m$. What is the optimal arrangement of workers for each equipment?

Each equipment is assigned with $n$ workers with one from each step. That is, for some $m$-permutations  $\sigma_1,\dots,\sigma_n$,  workers $(\sigma_1(i),1) ,\dots,(\sigma_n(i),n)$ are assigned to the $i$-th equipment, and hence $t_i=\sum_{j=1}^n a_{\sigma_j(i),j}$. The problem is
\begin{equation}\label{eq:opt}\mbox{ to minimize~~} T=  \max\left\{\sum_{j=1}^n a_{\sigma_j(i),j}:i=1,\dots,m\right\} \mbox{~over all $\sigma_1,\dots,\sigma_n\in S_m$,}\end{equation}
 where $S_m$ is the set of $m$-permutations. This optimization target is a minimax, and is very often consistent with variance reduction problems: to minimize quantities such as the sample variance of $t_1,\dots,t_n$.

Intuitively, the optimal arrangement should be such that $t_1,\dots,t_m$ are close enough, and ideally equal.
Since $$t_1+\dots+t_m=\sum_{i=1}^m\sum_{j=1}^n a_{i,j},$$ we have that $$T^*=\min_{\sigma_1,\dots,\sigma_n\in S_m}T \ge \frac 1m\sum_{i=1}^m\sum_{j=1}^n a_{i,j},$$ and ideally, $T^*$ would be almost equal to $\frac 1m\sum_{i=1}^m\sum_{j=1}^n a_{i,j}$.
Here we have two general questions:
\begin{enumerate}[(i)]
\item What is an optimal arrangement $(\sigma_1,\dots,\sigma_n)$ for \eqref{eq:opt}? How could we calculate an optimal arrangement $(\sigma_1,\dots,\sigma_n)$?
\item Under what conditions, $T^*=\frac 1m\sum_{i=1}^m\sum_{j=1}^n a_{i,j}$? How could we calculate $T^*$?
\end{enumerate}
Both questions are related to the concept of joint mixability, the main focus of this paper.

\subsection{Definitions and terminologies}
 Throughout, $n$ and $d$ are positive integers, and we assume a  atomless space $(\Omega, \mathcal A, \p)$ of random variables taking values in a semigroup $\G$ which can be chosen as $\R^d$  in most  cases.  In the literature, complete mixability and joint mixability are defined for distributions on $\R$. Theoretically, the concepts of mixability do not require any extra mathematical (topological, algebraic, analytical) structure on the underlying set $\G$ of study, other than an addition ($+$); in view of applications, only the case $\G=\R^d$ is particularly relevant.
  In the following we use the term ``distributions" for probability measures.

\begin{definition}[Joint mixability]\label{def1}
An $n$-tuple $(F_1,\cdots,F_n)$ of  distributions on $\G$ is \emph{jointly mixable} (JM) if there exists a distribution $H$ on $\G^n$ with margins $F_1,\dots,F_n$ such that  $H$ is supported in
$\{(x_1,\dots,x_n)\in \G^n: x_1+\dots+x_n=K\}$  for some $K\in \G$.
\end{definition}

In the literature, mixability is often defined  using the language of random variables. The two definitions are equivalent. 

\begin{definition}[Definition given in \cite{WPY13}]
A random vector $(X_1,\cdots,X_n)$ satisfying \begin{equation}\label{eq:defcm}
X_1+ \dots+X_n=K \mbox{~~for some~~} K \in \G,
\end{equation}
is called a \emph{joint mix}.
An $n$-tuple $(F_1,\cdots,F_n)$ of  distributions on $\G$ is \emph{jointly mixable} (JM) if
there exists a joint mix with marginal distributions $F_1,\cdots,F_n$. When $\G=\R^d$,  $K$ in \eqref{eq:defcm} is called a \emph{joint center} of $(F_1,\cdots,F_n)$.
\end{definition}

Joint mixability is supported by many applications, including the optimization problem in Section \ref{sec:opt}; see also Section \ref{sec:app} below.
Suppose that for $j=1,\dots,n$, $F_j$ is a discrete distribution on $\R$ supported on distinct points $a_{1,j},\dots,a_{m,j}$ with point-mass $1/m$ each. Let us recall question (ii) in Section \ref{sec:opt}. If $T^*=\frac 1m\sum_{i=1}^m \sum_{j=1}^n a_{i,j}$, then for some $\sigma_1^*,\dots,\sigma_n^*\in S_m$,
$$T^*= \sum_{j=1}^n a_{\sigma^*_j(i),j}, ~~~ i=1,\dots,m. $$
Now let $U$ be a discrete uniform random variable in $\{1,\dots,m\}$, and define
$X_j=a_{\sigma^*_{j}(U),j}$, $j=1,\dots,n$. It follows that $X_j\sim F_j$, $j=1,\dots,n$, and $X_1+\dots+X_n=T^*.$ That is, $(F_1,\dots,F_n)$ is jointly mixable, and the optimal arrangement $(\sigma_1^*,\dots,\sigma_n^*)$ in question (i) corresponds to a joint mix $(X_1,\dots,X_n)$. Question (ii) in Section \ref{sec:opt} is a special question of joint mixability.


Below we give the definition of complete mixability, which is the homogeneous case of joint mixability when all marginal distributions are identical.

\begin{definition}[Complete mixability]\label{def3}
  A   distribution $F$ on $\G$ is called \emph{$n$-completely
mixable} ($n$-CM) if the $n$-tuple $(F,\dots,F)$ is jointly mixable.  When $\G=\R^d$, $\mu=K/n$  is called a \emph{center} of $F$, where $K$ is the joint center of the $n$-tuple $(F,\dots,F)$.  A joint mix with identical margins $F$   is called a \emph{complete mix}.
\end{definition}

The reason why distributions in Definitions \ref{def1}-\ref{def3} are called \emph{mixable} is that we are curious about whether one is \emph{able} to find a joint \emph{mix} with the given constraints on margins.

Although complete mixability is a special case of joint mixability, the two concepts are studied separately in the literature as they require mathematical techniques at significantly different levels; see for example the  results on monotone densities in \cite{WW11} and \cite{WW14}. In addition, $n$-complete mixability is a property of a single distribution, allowing us to study the property by letting $n$ vary.

We denote by $\M_n(\mu)$ the set of
all $n$-CM distributions on $\R$ with center $\mu$, and
  by $\mathcal J_n(K)$ the set of all $n$-tuples of JM distributions with joint center $K$, that is,
$$\mathcal J_n(K)=\{(F_1,\cdots,F_n): (F_1,\cdots,F_n) \mbox{~is JM  with joint center $K$}\}.$$
Apparently, $F\in \M_n(\mu)$ if and only if $(F,\cdots,F)\in \mathcal J_n(n\mu)$.
For $n=1$ or $n=2$, the sets $\mathcal J_n(K)$ and $\mathcal M_n(\mu)$ are fully characterized. For $n\ge 3$, a full characterization of either set is still an open question.

\begin{remark}
In Definitions \ref{def1}-\ref{def3}, both the summation $x_1+\dots+x_n$ and the constant constraint  $K$ in the support of  $H$ are for  mathematical tractability and practical relevance. Other constraints may be chosen for different purposes of applications or theoretical studies.
\end{remark}

\subsection{Related literature}

In this section we provide a non-exhaustive brief list of related literature on complete and joint mixability, especially for the reader who is new to this topic.   Except for a few early milestone studies, most papers listed here are within the recent few years.

 Probability measures with given margins have been studied since the early work by Fr\'echet \cite{F51} and Hoeffding \cite{H40}; see also the milestone papers \cite{S65, T80}.  The first study of questions specifically related to complete mixability was given in \cite{GR81} where uniform distributions were shown to be $n$-CM for $n\ge 2$. Relevant contributions from  the perspectives of mass transportation, variance reduction and stochastic orders can be found in \cite{R82, RU02, MS02, KS06}. The terms \emph{complete mixability} and \emph{joint mixability} were introduced in \cite{WW11} and \cite{WPY13} respectively, along with properties and results on the complete mixability of monotone densities. Recent advances on complete and joint mixability can be found in \cite{PWW12, PWW13, WW14}.

As opposed to the strongest positive dependence (see for instance \emph{comonotonicity} in \cite{DDGKV02}), a universal notion of the strongest negative dependence does not exist for a collection of more than two random variables, and the corresponding optimization problems are generally much more complicated than those involving the strongest positive dependence. Complete and joint mixes are sometimes argued to have the strongest negative dependence structure as they naturally solve a large class of optimization problems. Recent studies searching for a notion of extremal negative dependence can be found in \cite{DD99, WW11, CL13,  CL14, LA14}.
 A recent review  on extremal dependence concepts is given in \cite{PW14b}.

Algorithms related to   mixability have been designed for questions (i) and (ii) in Section \ref{sec:opt}. 
An early study on \emph{rearrangement methods} is found in \cite{R83}; some recent research includes \cite{PR12a, EPR13, PW14a, H14}. In particular, \cite{H14} showed that question (i) in Section \ref{sec:opt} is NP-complete even in the case when all $a_{i,j}\in \mathbb Z$. As such, an analytical characterization of joint mixability is of considerable importance.

\subsection{Applications}\label{sec:app}

The concepts of complete and joint mixability are closely related to many optimization problems with marginal constraints. We discuss a few of them in this section. For the reader who is only interested in mathematical challenges, this section may be skipped.

Let $\mathcal X$ be a convex cone of random variables (taking values in $\R$) of interest; $\mathcal X$ can be chosen as $L^1$ (the set of integrable random variables) or $L^\infty$ (the set of bounded random variables) in most applications.
For some univariate distributions $F_1,\dots,F_n$, define the \emph{aggregation set}
$$\mathcal D_n=\{X_1+\dots+X_n: X_i\in\mathcal X,~X_i\sim F_i,~ i=1,\dots,n\}\subset \mathcal X.$$
Many optimization problems, including variance reduction problems, convex functionals minimization, and maximin and minimax problems such as \eqref{eq:opt} described in Section \ref{sec:opt}, boil down to the search for the smallest element in $\mathcal D_n$ with respect to convex order.
\begin{definition}[Convex order] \label{def:cxod}
Let $X$ and $Y$ two random variables with finite mean. $X$ is smaller than $Y$ \emph{in convex order}, denoted by $X\lcx Y$, if for all convex functions $f$,
\begin{equation}\E[f(X)]\le \E[f(Y)],\label{convexorder}\end{equation}
whenever both sides of \eqref{convexorder} are well-defined.
\end{definition}

It is well-known that the convex ordering largest element in $\mathcal D_n$
is always obtained by $F_1^{-1}(U)+\cdots+F_n^{-1}(U)$ for a random variable $U \sim \U[0,1]$.
However, it remains open in general to find the smallest element  in $\mathcal D_n$ with respect to convex order for $n\ge 3$.  \cite{BJW14} gave an example where $\mathcal D_n$ does not contain a smallest element in this sense.

When $(F_1,\dots,F_n)$ is JM with joint center $K$, it is easy to see that $K\in \mathcal D_n$ and $K\lcx S$ for all $S \in \mathcal D_n$.
In \cite{WW11} and \cite{EHW14},  the smallest element with respect to convex order in $\mathcal D_n$ is characterized based on joint mixability when $F_1,\cdots,F_n$ have monotone densities even if $(F_1,\dots,F_n)$ is not jointly mixable.

The concepts of complete and joint mixability have also raised a considerable interest in quantitative risk management, as it plays an important role in the context of \emph{risk aggregation with dependence uncertainty}.
A typical question in this field concerns the calculation of
\begin{equation}\label{eq:qrm}\sup\{\rho(S):S\in \mathcal D_n\} \mbox{~~and~~} \inf\{\rho(S):S\in \mathcal D_n\}\end{equation}
for a law-determined risk measure $\rho: \mathcal X\to \R$; see the book \cite[Section 6.2]{MFE05} and the early work on the risk measure Value-at-Risk (VaR) in \cite{EP06}. Here the set $\mathcal D_n$ represents the set of possible aggregate risks under model uncertainty at the level of dependence, a common setup in risk management practice.

In the case when $\rho$ is a convex risk measure, $\rho$ typically respects convex order; see  for instance \citet[Section 4]{FS11}.  Thus \eqref{eq:qrm}  boils down to questions of convex order in $\mathcal D_n$ as discussed above. Convex risk measures include the Expected Shortfall, a popular risk measure used in banking regulation; see \cite{EPRWB14}.

Among non-convex risk measures, the Value-at-Risk (VaR) is of particular interest in portfolio management. The Value-at-Risk of a random variable $X$ at level $p\in (0,1)$ is defined as the (left-continuous) inverse distribution function
$$\VaR_{p}(X)=\inf \{x\in \R: \p(X\le x)\ge p\}.$$
Quantities of interest are
\begin{equation} \label{eq:qrm1}
\WVaR_p=\sup\{\VaR_{p}(S):S\in \mathcal D_n\},
 \mbox{~~and~~}
 \BVaR_p=\inf\{\VaR_{p}(S):S\in \mathcal D_n\},~~p\in (0,1).
\end{equation}
    For more discussions and applications of this topic, see \cite{EPR13, EPRWB14}.
 The following result on $\WVaR_p$ is given in \cite{WPY13}; the case of $\BVaR_p$ is symmetric.
\begin{enumerate}[(1)]
\item For each $p\in (0,1)$, let $\Phi(p)=\frac{1}{1-p}\sum_{i=1}^n \int_p^1 \VaR_q(X)\d q.$ It holds that
          $\WVaR_p\leq
           \Phi(p).$

\item  For each $p\in (0,1)$, the equality
$\WVaR_p=
           \Phi(p)$
            holds if
           and only  if the $n$-tuple of  the distributions of
           $F_1^{-1}(W),\cdots,F_n^{-1}(W)$ is jointly mixable, where $W\sim
           \mathrm{U}[p,1]$.
           \end{enumerate}
The above result can also be applied to find minimal or maximal probability function of random variables in $\mathcal D_n$.

Although in many cases $(F_1,\dots,F_n)$ is not jointly mixable,  solutions of \eqref{eq:qrm} and \eqref{eq:qrm1} can still be obtained based on conditional complete or joint mixability in many cases; see  \cite{WPY13, PR13, EPR13,  BJW14, EHW14} for work in this direction. Some other recent research on  \eqref{eq:qrm} and \eqref{eq:qrm1} involving mixability can be found in \cite{PR12b, CV13, PR14, PWW13, BRV13, AP14, EWW14}.
 We refer to  \cite{EPRWB14} for a recent review of this subject  in the context of banking regulation, and the book \cite{R13} contains a comprehensive treatment of many related problems.

\section{Current open questions}

In this section, we discuss some open questions in complete and joint mixability. Unless otherwise specified, we consider $\G=\R$,  and $F$ is a distribution on $\R$. We use $L^0$ for the set of all random variables in $(\Omega,\mathcal A, \p)$ taking values in $\R$, and we use $\id_{A}$ to denote the indicator function of a set $A$. It is not necessary to read the following questions in a particular order.

\subsection{Uniqueness of the center}

Suppose that $F$ is $n$-CM. It is obvious that if $F$  has finite mean $\mu$, then its center
is unique and equal to $\mu$.  It is shown that if $x\p(|X|\ge x)\rightarrow 0$ as $x\rightarrow \infty$ for $X\sim F$, then the center of $F$ is also unique; see \cite[Proposition 2.1]{WW11}. This uniqueness can be easily extended to the case of $\G=\R^d$.  For a generic Abelian group $\G$, the uniqueness is not guaranteed; an example can be easily built for finite cyclic groups. For instance, consider a Bernoulli distribution $\mathrm{Bern}(1/2)$ on $\mathbb Z_2$ with $\p(X=0)=\p(X=1)=1/2$ for $X\sim \mathrm{Bern}(1/2)$. It is obvious that $X+X=0$ and $X+(1-X)=1$ on $\mathbb Z_2$, hence the center is not unique in this setting.

We are interested in whether the center $\mu$ is always unique for the case $\G=\R$ or $\R^d$.
Non-uniqueness may only happen in the case that the support of $F$ is unbounded from both sides, and $F$ does not have finite mean.
Note that the index $n$ in complete mixability is irrelevant; indeed if a distribution $F$ is $n$-CM with center $\mu_1$ and $k$-CM with center $\mu_2$, $\mu_1\ne \mu_2$, then $F$ is also $nk$-CM with centers $\mu_1$ and $\mu_2$.  Therefore, it suffices to determine whether a distribution can be $n$-CM with different centers for any $n\in \N$.

\begin{problem}
Is the center of mixability always unique for a distribution on $\R$ (or $\R^d$)? In other words, for $\mu,\nu\in \R$, $\mu\ne \nu$, is it true that $\mathcal M_n(\mu)\cap\mathcal M_n(\nu)=\varnothing$?
\end{problem}

\subsection{Generic proofs of some theorems}\label{sec:23}

Below we list some main results on complete and joint mixability in the recent literature.
\begin{enumerate}[(a)]
\item \citep{RU02} Any continuous distribution function $F$ having a symmetric and
unimodal density is $n$-CM for $n\ge 2$.
\item \citep{WW11} Suppose that $F$ admits a monotone density on its essential support $[a,b]$ with mean $\mu$. Then $F$ is $n$-CM if and only if
\begin{equation}\label{meancd}a+\frac{b-a}n\le \mu \le b-\frac{b-a}n.\end{equation}
\item  \citep{PWW12} Suppose that $F$ admits a concave density on its essential support.
Then $F$ is $n$-CM for $n\ge 3$.
\item \citep{PWW13} Suppose that $F$ admits a density $f$ on a finite interval $[a,b]$, and $f(x)\ge \frac3{n(b-a)}$ on $[a,b]$. Then  $F$ is $n$-CM.
\item \citep{WW14} Suppose that $F_1,\dots,F_n$ all admit increasing (or decreasing) densities on their essential supports $[a_i,b_i]$ and have mean $\mu_i$, $i=1,\dots,n$, respectively. Then $(F_1,\dots,F_n)$ is JM if and only if
\begin{equation}\label{meancd3}\sum_{i=1}^n a_i+\max_{i=1,\dots,n} (b_i-a_i) \le \sum_{i=1}^n \mu_i \le\sum_{i=1}^n b_i-\max_{i=1,\dots,n} (b_i-a_i) .\end{equation}
\item \citep{WW14}
Suppose that $F_i\sim E_1(\mu_i,\sigma_i^2, \phi)$, where $E_1$ is the 1-elliptical distribution (for definition, see \cite{FKN90}) with parameters $\mu_i\in \R$, $\sigma_i\ge 0$, $i=1,\cdots,n$, and $\phi$ is a characteristic generator for an $n$-elliptical distribution.   Then $(F_1,\dots,F_n)$ is JM  if and only if
\begin{equation}\sum_{i=1}^n \sigma_i\ge 2\max_{i=1,\cdots,n} \sigma_i.\label{varcd}\end{equation}
\end{enumerate}

Note that all results (a)-(e) include uniform distributions as a special case. The proofs of (a) and (f) are analytical and reasonably straightforward due to the symmetric nature of the underlying distributions. The dependence structure of a corresponding  joint mix in (a) and (f) is clear.

However, the proofs of the recent results on complete mixability, namely (b)-(d), are all based on combinatorics and discretization of distributions. We outline the common logic of the proofs as follows. To show that a distribution $F$ is $n$-CM, first we find a sequence of discretizations of this distribution, say $F_N$, with $F_N\to F$ (sufficiently in the weak sense) as $N\rightarrow \infty$. Then, for a fixed $N$, we try to show that $F_N$ can be decomposed to a convex combination of $n$-discrete uniform distributions with the same mean, or a convex combination of known-to-be-$n$-CM discrete distributions with the same mean. This often involves mathematical induction on the number of points in the support of $F_N$. The proof of result (e) in \cite{WW14} is even more complicated; it involves decomposition of $F_1,\dots,F_n$ into combination of distributions with step density functions (which are not jointly mixable, but in some sense close to being jointly mixable), and a mathematical induction on the number of effective steps is used.
 The proofs for the above-mentioned results are typically very long and technical, and more importantly the details of the dependence structure for a joint mix are always unclear. These rather unfortunate features significantly reduce the accessibility of the theory of mixability for the general reader.

 Through private communications with many scholars interested in this topic, the author believes that   generic (probabilistic, analytic) proofs without involving combinatorics or mathematical induction is in demand for the future development of the theory.

\begin{problem}
Is there a generic (probabilistic, analytic) proof of the main results in complete and joint mixability?
\end{problem}

Some duality theorems on probability measures with given margins in the literature can be applied to complete and joint mixability.  Recent studies on complete mixability using duality methods are found in \cite{PR12b, PR13, W14}. The following theorem was essentially established in \cite{S65, R82}. How they could be used to generate new results on mixability is still unclear.
\begin{theorem}[\cite{S65, R82}] \label{dual}
For distributions $F_1,\dots,F_n$ on $\R$, the following statements  are equivalent:
\begin{enumerate}[(i)]
\item
$(F_1,\dots,F_n)$ is jointly mixable with joint center $K$.
\item  For all measurable functions $f_i:\R\to\R$, $i=1,\dots,n$,
$$\sum_{i=1}^n \int f_i \d F_i \ge \inf\left\{\sum_{i=1}^n \E[f_i(Y_i)]: Y_1,\dots,Y_n\in L^0,~\sum_{i=1}^n Y_i=K \right\},$$
whenever both sides of the above equation are finite.
\item  For all measurable functions $f_i:\R\to\R$,  $i=1,\dots,n$ such that $\sum_{i=1}^n f_i(x_i)\ge \id_{\{x_1+\dots+x_n=K\}}$ for all $(x_1,\dots,x_n)\in \R^n$,
$$\sum_{i=1}^n \int f_i \d F_i \ge 1,$$
whenever the left-hand side of the above equation is finite.
\end{enumerate}
\end{theorem}

\begin{proof} ~
\begin{enumerate}[(a)]
\item (i)$\Rightarrow$(iii): Let $(X_1,\dots,X_n)$ be a joint mix with joint center $K$, and $X_i\sim F_i$, $i=1,\dots,n$. Then for measurable functions $f_1,\dots,f_n$ in (iii),
$$\sum_{i=1}^n \int f_i \d F_i=\sum_{i=1}^n \E[f_i(X_i)]\ge \E[\id_{\{X_1+\dots+X_n=K\}}]=1.$$
\item (iii)$\Rightarrow$(ii): 
For measurable functions $f_i:\R\to\R_+$, $i=1,\dots,n$, let $$\xi=\inf\left\{\sum_{i=1}^n \E[f_i(Y_i)]:  Y_1,\dots,Y_n\in L^0,~\sum_{i=1}^n Y_i=K \right\}.$$
It follows that
$$\xi\le \inf\left\{\sum_{i=1}^n  f_i(y_i):  y_1,\dots,y_n\in \R,~\sum_{i=1}^n y_i=K \right\}.$$
That is,
$\sum_{i=1}^n f_i(x_i)/\xi \ge \id_{\{x_1+\dots+x_n=K\}}$ for all $(x_1,\dots,x_n)\in \R^n$. It follows from (iii) that
\begin{equation}\label{eq:S65}\sum_{i=1}^n \int f_i \d F_i\ge \xi= \inf\left\{\sum_{i=1}^n \E[f_i(Y_i)]: Y_i\in L^0,~i=1,\dots,n,~\sum_{i=1}^n Y_i=K \right\}. \end{equation}
Now we have shown that \eqref{eq:S65} holds for non-negative functions $f_1,\dots,f_n$. Note that \eqref{eq:S65} is invariant under a shift in any of $f_1,\dots,f_n$, and hence it holds also for all functions $f_1,\dots,f_n$ bounded from below. For functions that are unbounded from below, a standard approximation argument using monotone convergence theorem would show that \eqref{eq:S65} still holds.
\item (ii)$\Rightarrow$(i): this directly follows from Theorem 7 of \cite{S65}.  \qedhere
\end{enumerate}
\end{proof}

\begin{remark}
Indeed, in Theorem \ref{dual}, (i)$\Leftrightarrow$(ii) can be obtained from a particular case of \cite[Theorem 7]{S65}, and
(i)$\Leftrightarrow$(iii) can be obtained from a    particular case  of \cite[Equation (4)]{R82}; see also \cite[Theorem 1]{RR95}. None of the results in \cite{S65} and \cite{R82} are stated specifically for the case of mixability.
\end{remark}


\subsection{Representation and decomposition}

There are two decompositions of complete and joint mixability into simple objects, as shown in Theorems \ref{th:dist} and \ref{th:rep} below. Although similar ideas may be found in the literature,  the theorems themselves are new in this paper.

In the following, we say a distribution $F$ is an \emph{$n$-discrete uniform distribution} on $(a_1,\dots,a_n)\in \R^n$ if $\p(X=x)=\#\{i=1,\dots,n:a_i=x\}/n$ for $X\sim F$.
Theorem 3.2 of \cite{PWW12} says that a discrete distribution $F$ is $n$-CM with center $\mu$ if and only if it has a decomposition:
$$F=\sum_{i=1}^\infty b_i F_i,$$
where $\sum_{i=1}^\infty b_i=1,$ $b_i\ge 0$, $i\in \N$ and $F_i,~i\in \N$ are $n$-discrete uniform distributions with mean $\mu$.
A stronger result can be obtained for any CM distributions.

\begin{theorem}\label{th:dist}
A distribution $F$ on $\R$ is $n$-CM with center $\mu$ if and only if it has the following representation
\begin{equation} \label{eq:rep}F=\int_{\R^n} F_\mathbf{a}\d h(\mathbf{a}),\end{equation}
where $F_\mathbf{a},~\mathbf{a} \in \R^n$ are $n$-discrete uniform distributions with mean $\mu$,  $h$ is a probability measure on $\R^n$, and for a fixed $x\in \R$, $F_\mathbf{a}(x)$ is  measurable in $\mathbf{a}\in \R^n$.
\end{theorem}
\begin{proof}
Let $\mathbf X_\mathbf{a}$ be a $\mu$-centered complete mix with identical marginal distributions $F_\mathbf{a}$, $\mathbf{a} \in \R^n$. Take a random vector $\mathbf A\sim h$ be independent of $\mathbf X_\mathbf{a}$, $\mathbf{a} \in \R^n$ and define $$\mathbf X_{\mathbf  A}(\omega)=\mathbf X_{\mathbf A(\omega)}(\omega),~ \omega \in \Omega.$$ It is easy to see that
$\X_{\mathbf A}$  is also a $\mu$-centered complete mix. The marginal distribution of $\mathbf X_{\mathbf A}$ can be easily calculated as $$\p\left(\X_{\mathbf A}\le (x,\infty,\dots,\infty)\right)=\int_{\R^n} \p\left(\X_\mathbf{a}\le  (x,\infty,\dots,\infty)\right) \d h(\mathbf{a})=\int_{\R^n} F_\mathbf{a}\d h(\mathbf{a})=F(x),~~x\in \R.$$
Hence, $F$ is $n$-CM with center $\mu$.

Now suppose that $F$ is $n$-CM with center $\mu$. Let $\X=(X_1,\dots,X_n)$ be a $\mu$-centered  complete mix with identical marginal distributions $F$.
Let $F_{\mathbf a}$ for $\mathbf a  \in \R^n,~ \mathbf a \cdot \mathbf 1_n=n\mu$ be the $n$-discrete uniform distribution on $\mathbf a$ with mean $\mu$. It is obvious that for fixed $x\in \R$, $F_\mathbf{a}(x)$ is  measurable in $\mathbf{a}\in \R^n$.
Let $U$ be a discrete uniform distribution on $(1,\dots,n)$, independent of $\X$, and $Z=\sum_{i=1}^n X_i\id_{\{U=i\}}$. It is straightforward to verify that that $Z\sim F$, and $$\p(Z\le x|\X)=\frac 1n\sum_{i=1}^n \p(X_i\le x|\X)=\frac{1}{n}\sum_{i=1}^n\E[\id_{\{X_i\le x\}}|\X]=\frac{1}{n}\sum_{i=1}^n\mathcal \id_{\{X_i\le x\}}=F_\mathbf a(x)\big{|}_{\mathbf a =\X}, ~~x\in \R.$$
It follows that
$$F(x)=\E[\p(Z\le x|X)]=  \int_{\R^n} F_\mathbf a(x) \d \p(\X\le \mathbf a),~~x\in\R,
$$
and $h(\mathbf a)$ in \eqref{eq:rep} can be chosen as $\p(\X\le \mathbf a)$.
\end{proof}


Since complete (and joint) mixability is preserved by taking weak limit (see \cite{WW11}), it is often sufficient to investigate complete mixability for bounded discrete distributions on $\mathbb Z$ and then take a limit for general distributions; this technique was used repeatedly in \cite{WW11,PWW12,PWW13}.
We say a vector $\X$ is a \emph{binary multinomial random vector} if  $\X$ has a multinomial distribution with the ``number of trials" parameter $n=1$, that is, $\X$ takes values in $\{0,1\}^n$ and exactly one of the components of $\X$ is 1. The following decomposition, which could be seen as ``perpendicular" to Theorem \ref{th:dist}, may  be of help to characterize complete and joint mixability.

\begin{theorem}\label{th:rep}
Suppose that $\X$ takes values in $\mathbb Z_+$.
$\X$ is a joint mix with joint center $N\in \mathbb Z_+$ if and only if it has the following representation
\begin{equation} \label{eq:rep}\X=\sum_{k=1}^{N} \X_k,   \end{equation}
where $\X_k$, $k=1,\dots,N$ are binary multinomial  random vectors.
\end{theorem}
\begin{proof}
Suppose that \eqref{eq:rep} holds. Since $\X_k \cdot \mathbf 1_n=1$, it is easy to see that $\X$ is a joint mix with center $N$.
Now suppose that $\X=(X_1,\dots,X_n)$ is a joint mix. For $k=1,\dots, N$ and $i=1,\dots,n$ let
$$Y_{k,i}=\id_{\{\sum_{j=1}^i X_j\ge k\}}-\id_{\{\sum_{j=1}^{i-1} X_j\ge k\}}$$
with the convention that $\sum_{j=1}^{0} X_j=0$,
and let $\X_k=(Y_{k,1},\dots,Y_{k,n})$.
Since $X_1+\dots+X_n=N$, each $\X_k$ is binary multinomial.
Then
\begin{align*}
\sum_{k=1}^N \X_k=&\left(\sum_{k=1}^N Y_{k,1},\dots,\sum_{k=1}^N Y_{k,n}\right)\\
=&\left(\sum_{k=1}^N \id_{\{X_1\ge k\}},\sum_{k=1}^N \id_{\{X_1+X_2\ge k\}}-\sum_{k=1}^N\id_{\{X_1\ge k\}},\dots,\sum_{k=1}^N \id_{\{\sum_{j=1}^n X_j\ge k\}}-\sum_{k=1}^N\id_{\{\sum_{j=1}^{n-1} X_j\ge k\}}\right)\\
=&\left(X_1,X_1+X_2-X_1,\dots, \sum_{j=1}^n X_j-\sum_{j=1}^{n-1} X_j\right)=(X_1,X_2\dots,X_n).
\end{align*}
Thus, $\X$ admits a decomposition of \eqref{eq:rep}.
\end{proof}

As a trivial consequence of Theorem \ref{th:rep}, any binomial distribution with parameters $(n,1/n)$ for $n\in \N$ is $n$-CM, since it is the marginal distribution of   multinomial distribution with parameters $(n;1/n,\dots,1/n)$, and any multinomial random vector has a natural representation \eqref{eq:rep}.

Arguments of the type of Theorem \ref{th:dist} has been applied extensively in the recent literature to show the complete/joint mixability of some classes of distributions. It remains a question whether Theorem \ref{th:rep} can be useful in a non-trivial way.
\begin{problem}
Is Theorem \ref{th:rep} helpful (and how) in characterizing more classes of CM and JM distributions?
\end{problem}

\subsection{Norm condition}

Below we discuss the relationship between mixability and law-determined norms.
First we give the definition of a law-determined norm.

\begin{definition}[Law-determined norm]\label{def:norm}
A law-determined  norm $||\cdot||$ on $L^0$ maps $L^0$ to $[0,\infty]$, such that
\begin{enumerate}[(i)]
\item $||aX||=|a|\cdot||X||$ for $a\in \R$ and $X\in L^0$;
\item $||X+Y||\le ||X||+||Y||$ for $X,Y \in L^0$;
\item $||X||=0$ implies $X=0$ a.s.;
\item $||X||=||Y||$ if $X\laweq Y$, $X,Y \in L^0$;
\item $||X||\le ||Y||$ if $0\le X\le Y $ a.s.
\end{enumerate}
\end{definition}
The $L^p$-norms $p\in [1,\infty)$, $||\cdot||_p: L^0 \rightarrow [0,\infty],$  $X\mapsto (\E[|X|^p])^{1/p}$  and the $L^\infty$-norm $||\cdot||_\infty: L^0 \rightarrow [0,\infty],$  $X\mapsto \esssup(|X|)$  are law-determined norms. Here, we allow $||\cdot||$ to take a value of $\infty$, which means that the non-negative functional $||\cdot||$ is not necessarily a norm in the common sense; we slightly abuse the terminology here since all natural examples are norms in  respective proper spaces.
We obtain a necessary condition for complete and joint mixability based on law-determined norms. In what follows,  $(\cdot)_+=\max\{\cdot, 0\}$ and $(\cdot)_-= - \min\{\cdot, 0\}$.

\begin{theorem}\label{norm}
Suppose that $(F_1,\dots,F_n)$ is JM with joint center $K$,  $X_i\sim F_i$, $i=1,\dots,n$, $||\cdot||$ is any law-determined norm.  Then we have that \begin{equation} ||(X_i-\mu_i)_+||\le  \sum_{j=1,j\ne i}^n||(X_j-\mu_j)_-|| ~~~\mbox{and}~~~||(X_i-\mu_i)_-||\le  \sum_{j=1,j\ne i}^n ||(X_j-\mu_j)_+||,\label{normineq}\end{equation}
for all $i=1,\dots,n$, and all $\mu_1,\dots,\mu_n\in \R$ such that $\mu_1+\dots+\mu_n=K$.
\end{theorem}

\begin{proof}
Since $(F_1,\dots,F_n)$ is JM, there exist random variables $X_1\sim F_1,\dots,X_n\sim F_n$ such that
$X_1+\dots+X_n=K$. It follows that
$X_1-\mu_1=-((X_2+\dots+X_n)-(\mu_2+\dots+\mu_n))$ and hence
 \begin{equation} (X_1-\mu_1)_+=\left(\sum_{i=2}^n (X_i-\mu_i)\right)_-\le \sum_{i=2}^n (X_i-\mu_i)_-.\label{normineq2}\end{equation}
Applying $||\cdot||$ on both sides of \eqref{normineq2}, we obtain
$$||(X_1-\mu_1)_+|| = \left|\left|\left(\sum_{i=2}^n (X_i-\mu_i)\right)_-\right|\right|  \le \left|\left|\sum_{i=2}^n (X_i-\mu_i)_-\right|\right| \le \sum_{i=2}^n \left|\left|(X_i-\mu_i)_-\right|\right|  . $$
The rest parts are obtained by symmetry.
\end{proof}

A  similar version of Theorem \ref{norm} for complete mixability is listed below.

\begin{theorem}\label{norm2}
Suppose that $F$ is $n$-CM with center $\mu$, $X\sim F$ and $||\cdot||$ is any law-determined norm.  Then we have that \begin{equation} ||(X-t)_+||\le (n-1) ||(X-s)_-|| ~~~\mbox{and}~~~||(X-t)_-||\le (n-1)  ||(X-s)_+||,\label{normineq'}\end{equation}
for all $t\in \R$ and $s=(n\mu-t)/(n-1)$.
\end{theorem}

It is worth noting that if we take $||\cdot||=||\cdot||_\infty$ and $s=t=\mu$ in Theorem \ref{norm2}, then we obtain that
$$||(X-\mu)_+||_\infty\le (n-1) ||(X-\mu)_-||_\infty,$$
which is
$b-\mu \le (n-1)(\mu-a)$, where   $a=\sup\{t\in \R: F(t)=0\}$ and $b=\inf\{t\in \R: F(t)=1\}$.
Combining with the other inequality in \eqref{normineq'}, we obtain \begin{equation}\label{meancd1}a+\frac{b-a}n\le \mu \le b-\frac{b-a}n.\end{equation}
\eqref{meancd1} is the \emph{mean condition} obtained in \cite{WW11}, one of the key necessary conditions for complete mixability, and is also a sufficient condition if $F$ has monotone density, shown in \cite{WW11}; see also (b) in Section \ref{sec:23}.
If we take $||\cdot||=||\cdot||_2$ and $\mu_i=\E[X_i]$ in Theorem \ref{norm}, and assume  that each  $F_i$ is $\mathrm{N}(\mu_i,\sigma_i^2)$, $i=1,\dots,n$, then we obtain that
$$2\max_{i=1,\dots,n}\sigma_i\le \sum_{i=1}^n \sigma_i.$$
which is a sufficient condition for $F_1,\dots,F_n$ to be JM, shown in \cite{WPY13}; see also (f) in Section \ref{sec:23}.
Both examples indicate that special cases of \eqref{normineq} and \eqref{normineq'} may be sufficient for particular classes of distributions.


\begin{problem}
Suppose that $F$ has mean $\mu$, and \eqref{normineq'} holds for all law-determined norms $||\cdot||$  and all $t\in \R$. (With what extra conditions, possibly some smoothness conditions) is it sufficient for  $F$ to be $n$-CM?
\end{problem}
This problem induces another interesting question which is not directly related to mixability: how can we characterize all possible law-determined norms in Definition \ref{def:norm}?

\subsection{Mixability in vector spaces}

Most of the literature has a focus on complete and joint mixability on $\R$ for its relevance in applications. Clearly, concepts of mixability can be naturally generalized to distributions on $\R^d$. However, existing non-trivial results in the multi-dimensional setting are very limited; an early study in this direction can be found in \cite{RU02} .

A simple observation is listed below. Its proof is straightforward and omitted.
\begin{proposition}
Suppose that $F$ on $\R^d$ is $n$-CM. Then the projection of $F$ to any subspace of $\R^d$ is $n$-CM.
\end{proposition}

We have the following conjecture, with its rationale explained below.
\begin{problem}
Is a uniform distribution on a convex set $C\subset \R^{d}$ necessarily $n$-CM for all $n\ge d+1$?
\end{problem}
Obviously, the above conjecture is equivalent to say that a uniform distribution on a convex set $C\subset \R^{d}$ is $(d+1)$-CM.
The trivial cases $d=0$ and $d=1$  are explained below.  When $d=0$, $C$ degenerates to a singleton, on which a distribution is always $n$-CM for $n\ge 1$. When $d=1$, $C$ is an interval, and a uniform distribution on an interval is $n$-CM for $n\ge 2$; this was already shown in \cite{GR81}. When $d=2$, any projection of a uniform distribution on $C$ to a line has a concave density. \cite{PWW12} showed that a distribution with a concave density is $n$-CM for $n \ge 3$. Of course, this is not sufficient for such a distribution to be $n$-CM on $\R^d$. However, we wonder how this type of dimension reduction would help to characterize complete mixability.
\begin{problem}
Suppose that $F$ is a distribution on $\R^d$, $d>1$, and the projection of $F$ to any essential subspace of $\R^d$ is $n$-CM.  (With what extra conditions) is it sufficient  for $F$ to be $n$-CM on $\R^d$?
\end{problem}

\subsection{Asymptotic mixability}

Let $F$ be an arbitrary distribution with bounded support. It has been observed that \citep[e.g.][]{BRV13, W14} when $n$ is large, it is more likely that $F$ becomes $n$-CM. \cite{PWW13} showed that a distribution on a bounded interval $[a,b]$ with a positive density function $f(x)\ge 3/(n(b-a))$ is $n$-CM. As a consequence, any continuous distribution with a density bounded away from zero is $n$-CM for $n$ sufficiently large. It is left open to answer whether this condition of a density bounded away from zero can be removed.
\begin{problem}
Are all absolutely continuous distributions on a bounded interval $n$-CM for large enough $n$?
\end{problem}

\subsection{Copula of a complete mix}

The major results in \cite{WW11, PWW12, PWW13, WW14} are based on combinatorics and mathematical induction. The dependence structure hidden in the proofs are unclear. It was noted in \cite{WW11} that an explicit form of a copula (which is generally not unique) of a complete  or joint mix is very difficult to write down. Since complete and joint mixability naturally give bounds to many optimization problems, it would be nice to have a copula of a complete mix, or a sampling method for simulation. The  two questions are of course  very much related.

\begin{problem}
Suppose that $F$ satisfies one of the sufficient conditions (for instance, (b) in Section \ref{sec:23}) and hence is $n$-CM.
What is a possible copula of a $n$-complete mix  with margins $F$ (or a joint mix with given margins)?
\end{problem}

\begin{problem}
Suppose that $F$ satisfies one of the sufficient conditions (for instance, (b) in Section \ref{sec:23}) and hence is $n$-CM. Could we simulate sample from a $n$-complete mix  with margins  $F$?
\end{problem}

\subsection{Characterizing more classes of CM/JM distributions}

It is a general task to characterize more classes of CM/JM distributions with their corresponding necessary and sufficient conditions. One particular question often discussed  concerns the unimodal densities, as it is relevant to many optimization problems outlined in \cite{BJW14}. \cite{PW14a} gave counter-examples where the mean condition \eqref{meancd1} is not sufficient for the complete mixability of a distribution  with a unimodal density  on a bounded interval.
\begin{problem}
Under what extra conditions a unimodal distribution on a bounded interval is $n$-CM?
\end{problem}
This question is particularly relevant to optimization problems when one of the inequalities in the mean condition \eqref{meancd1} is attained by an equality, as noted in \cite{BJW14}. That is, the mean of the distribution exactly divides the support $[a,b]$ of the distribution into two parts with lengths $(b-a)/n$ and $(b-a)(n-1)/n$, respectively.

\subsection{Convex order problems}

When $(F_1,\dots,F_n)$ is not JM, it is generally not clear whether there exists an element $S_0\in \mathcal D_n$ such that $S_0\lcx S$ for all $S \in \mathcal D_n$, where $\lcx$ stands for convex order in Definition \ref{def:cxod}. A counter-example is given in \cite{BJW14} showing an aggregation set $\mathcal D_n$ does not necessarily contain a smallest element with respect to convex order. However, for all commonly used distributions $F_1,\dots,F_n$, $\mathcal D_n$ seems to contain such a smallest element, as shown either theoretically or numerically. For instance, if each $F_i$ has a decreasing density, $i=1,\dots,n$, then  a smallest element with respect to convex order in $\mathcal D_n$ can be obtained; this was shown in \cite{EHW14}. It remains unclear under what conditions such a smallest element exists.

\begin{problem}
What are necessary and sufficient conditions for $\mathcal D_n$ to contain a  smallest element with respect to convex order?
\end{problem}

\subsection{Characterizing the aggregation set}

The last question is a  general question concerning Fr\'echet classes.
We use the aggregation set $\mathcal D_n$ as in the previous problem, and define
 $\mathcal D_n^*=\{S/n: S\in \mathcal D_n\}.$
It is obvious that the joint mixability of $F_1,\dots,F_n$ is equivalent to the inclusion of a degenerate random variable in $\mathcal D_n$.
In the case when $F=F_1=F_2=\cdots$ and $F$ has finite mean,
 \cite{MW14} showed that $ \mathcal D_n^*$ has an upper limit of $\mathcal C_F=\{S:S \lcx X,~X\sim F\}$  as $n\to \infty$. However, it is also noted that for a finite $n$, $ \mathcal D_n^*\subset \mathcal C_F$ but is generally not equal to $\mathcal C_F$. The only fully-characterized classes of $\mathcal D_n$ are when $n=2$ and the marginal distributions are Bernoulli; see \cite{MW14}.

\begin{problem}
How can one characterize $\mathcal D_n$ (maybe for some simple marginal distributions)? That is, for a given distribution $G$, determine whether $S\in \mathcal D_n$ for some $S\sim G.$
\end{problem}
 This question summarizes all challenges in complete and joint mixability. It is generally open for all $n\ge 2$.

\subsection{Some other open questions}

We conclude this paper by some other questions that are beyond the expertise of the author. To avoid misleading the reader with the author's naivety and ignorance,  we simply list some possible directions.
\begin{enumerate}
\item Algorithms to determine whether some distributions are jointly mixable, or solving question \eqref{eq:opt} in Section \ref{sec:intro}: see for instance \cite{PR12a, EPR13, PW14a, H14}. The conditions under which the swapping algorithms in \cite{PR12a} converges are still unclear.  Interestingly, \cite{H14} showed that the determination of the joint mixability of different discrete uniform distributions on $\mathbb Z$ is NP-complete.
\item Mixability under higher-dimensional constraints: for fixed bivariate or higher-dimensional marginal distributions,  determine whether a joint mix exists and develop algorithms for numerical determination.  Note that even to justify the existence of a joint distribution with given multivariate margins is not easy; see for instance \cite{S65, J97, EP09}.
\item Other multivariate functions replacing the summation of random variables in the definition of mixability; see \cite{BP14}.
\item The influence of the algebraic structure of a semigroup $\mathbb G$ on complete and joint mixability defined on $\mathbb G$.
\end{enumerate}

\subsection*{Acknowledgement}
Many of the questions outlined in this paper emerged during private communications with Paul Embrechts, Taizhong Hu, Tiantian Mao, Giovanni Puccetti, Ludger R\"uschendorf,  Bin Wang and Jingping Yang. The author is grateful to them and many other researchers for their contributions in theories and applications related to this topic. The author would also like to thank the Editor and an anonymous referee for helpful comments on an earlier version of the paper.
The author acknowledges financial support from the Natural Sciences and Engineering Research Council of Canada (NSERC Grant No. 435844).

 \end{document}